\documentclass[psamsfonts,12pt]{amsart}
\usepackage[margin=1in]{geometry} 
\usepackage{amsmath,amssymb,amsthm, amscd}
\usepackage{amssymb,fge}
\usepackage{tikz-cd}

\usepackage{color}

\newcommand\Q{\mathbb{Q}}

\def\p{{\mathfrak{p}}}

\def\q{{\mathfrak{q}}}

\DeclareMathOperator{\PGL}{PGL}
\DeclareMathOperator{\car}{char}

\DeclareMathOperator{\Sur}{Sur}

\DeclareMathOperator{\disc}{disc}

\DeclareMathOperator{\cf}{cf}
\newcommand{\mysetminus}{\mathbin{\fgebackslash}}

\newenvironment{Gal}{{\rm Gal\,}}{}
\newenvironment{Res}{{\rm Res\,}}{}
\newenvironment{Disc}{{\rm Disc}}{}

\newenvironment{Aut}{{\rm Aut}}{}

\newtheorem{theorem}{Theorem}[section]

\newtheorem{lemma}[theorem]{Lemma}
\newtheorem{proposition}[theorem]{Proposition}
\newtheorem{proposition-definition}[theorem]{Proposition-Definition}

\theoremstyle{definition}

\theoremstyle{remark}

\usepackage[margin=1in]{geometry}  
\usepackage{graphicx}              
\usepackage{amsmath}               
\usepackage{amsfonts}              
\usepackage{amsthm}                
\usepackage{verbatim}              
\usepackage{indentfirst}           
\usepackage{underscore}
\usepackage{enumitem}
\usepackage{relsize}
\usepackage{varwidth}
\usepackage{mathtools} 
\usepackage{hyperref}
\hypersetup{
    colorlinks=true,
    linkcolor=blue,
    filecolor=magenta,      
    urlcolor=cyan,
}
\usepackage{amssymb}
\usepackage[all]{xy}
\makeatletter
\renewcommand*\env@matrix[1][*\c@MaxMatrixCols c]{%
  \hskip -\arraycolsep
  \let\@ifnextchar\new@ifnextchar
  \array{#1}}
\makeatother

\theoremstyle{remark}
\newtheorem{rmk}{Remark}
\theoremstyle{remark}

\newcommand*{\Scale}[2][4]{\scalebox{#1}{$#2$}}%

\newcommand{\Mod}[1]{\ (\textup{mod}\ #1)}     

\newcommand{\sep}{\operatorname{sep}} 

\title[Riccati equations and dynamics over function fields]{Riccati equations and polynomial dynamics over function fields} 
\author[Wade Hindes]{Wade Hindes}
\address{Department of Mathematics, The Graduate Center, City University of New York (CUNY); 365 Fifth Avenue, New York, NY 10016, USA}
\author[Rafe Jones]{Rafe Jones}
\address{Department of Mathematics, Carleton College; One North College Street, Northfield, MN 55057, USA}
\begin{document}
\begin{abstract} 
Given a function field $K$ and $\phi \in K[x]$, we study two finiteness questions related to iteration of $\phi$: whether all but finitely many terms of an orbit of $\phi$ must possess a primitive prime divisor, and whether the Galois groups of iterates of $\phi$ must have finite index in their natural overgroup $\Aut(T_d)$, where $T_d$ is the infinite tree of iterated preimages of $0$ under $\phi$. We focus particularly on the case where $K$ has characteristic $p$, where far less is known. We resolve the first question in the affirmative under relatively weak hypotheses; interestingly, the main step in our proof is to rule out ``Riccati differential equations" in backwards orbits. We then apply our result on primitive prime divisors and adapt a method of Looper to produce a family of polynomials for which the second question has an affirmative answer; these are the first non-isotrivial examples of such polynomials. We also prove that almost all quadratic polynomials over $\Q(t)$ have iterates whose Galois group is all of $\Aut(T_d)$. 
\end{abstract}
\maketitle
\renewcommand{\thefootnote}{}
\section{Introduction}
Let $K$ be a global field with ring of integers $\mathcal{O}_K$, let $V_K$ be a complete set of discrete valuations on $K$ (corresponding to valuation rings in $K$), and let $\phi(x)\in K(x)$ be a rational function of degree $d\geq2$. The map $\phi:\mathbb{P}^1\rightarrow\mathbb{P}^1$ and its iterates induce a discrete dynamical system on $\mathbb{P}^1$, and we let $\phi^n$ denote the $n$-fold iterate of $\phi$. 

Given a pair of points $a,b\in\mathbb{P}^1(K)$, one would like to know whether or not the $a$-shifted $n$th iterate of $b$,\, $\phi^n(b)-a$, has a primitive prime factor, that is, whether or not there is a prime dividing $\phi^n(b)-a$ that does not divide any lower order $a$-shifted iterates: $v\in V_K$ is called a \emph{primitive prime divisor} of $\phi^n(b)-a$ if \vspace{.05cm}
\[v(\phi^n(b)-a)>0\;\;\text{and}\;\; v(\phi^m(b)-a)=0\; \text{for all}\; 1\leq m\leq n-1.\] 
Of course, one expects that most terms in the sequence $\phi^n(b)-a$ have primitive prime divisors, and to measure the failure of this heuristic, we define the \emph{Zsigmondy set} of $\phi$ and the pair $(a,b)$ to be \vspace{.05cm} 
\[\mathcal{Z}(\phi,a,b):=\{n\;:\;\phi^n(b)-a\;\text{has no primitive prime divisors}\}.\]

Evidence suggests that $\mathcal{Z}(\phi,a,b)$ is finite, unless the tuple $(\phi,a,b)$ is special in some way; for instance, if $b$ has finite orbit under $\phi$, then $\mathcal{Z}(\phi,a,b)$ is infinite for arbitrary $a\in K$, and the same conclusion holds for $a=0$ and any $b$ in the case $\phi(x) = x^d$. If $K$ has characterisitc zero, then $\mathcal{Z}(\phi,a,b)$ is known to be finite in many cases \cite{Xander,Ingram, Ingram2,Holly}, with the strongest results coming over function fields \cite{ABCimplies}.

On the other hand, there are very few results known for fields of positive characteristic. Nonetheless, in \cite{Trans,primdiv} the first author was able to prove the finiteness of Zsigmondy sets for polynomials $\phi$ over $\mathbb{F}_p(t)$ whenever there exist a pair of integers $(\ell,m)$ such that the curve 
\begin{equation}{\label{curve}}
C_{\ell,m}(\phi): Y^\ell=\phi^m(X)
\end{equation} 
is non-isotrivial, i.e. not defined over a finite field after a change of variables. In particular, such a result holds for most quadratic polynomials \cite[Corollary 1.2]{Trans}; in this case, $C_{2,2}(\phi)$ is an elliptic curve, and one can explicitly compute a $j$-invariant to detect isotriviality. However, until now it was not known whether or not this technique was more broadly applicable. 

In this paper, we use some properties of Riccati differential equations to find such a non-isotrivial pair $(\ell,m)$, allowing us to to prove a fairly general primitive prime divisors theorem for polynomials. Recall that a Riccati differential equation is one of the form $y' = f_0 + f_1y + f_2y^2$, where $y$ is an unknown function and $f_0, f_1, f_2$ are specified functions. As an application of this method and some recent work of Looper \cite{Looper} on the Galois groups of iterated trinomials, we construct explicit examples of non-isotrivial polynomials over function fields having large image arboreal Galois representations. 

In what follows $t$ is an indeterminate and $K/k(t)$ is a finite separable extension. In particular, one can extend the usual derivative $\frac{d}{dt}$ on $k(t)$ to $K$ via implicit differentiation; moreover, we let $\beta'$ denote the derivative of $\beta\in K$. \vspace{.0075cm}
\begin{theorem}{\label{thm:main}} Let $K/\mathbb{F}_q(t)$ be a function field, let $\phi(x)\in K[x]$ have degree $d\geq3$, and write 
\[\phi(x)=A_0x^d+A_{1}x^{d-1}+\dots A_{d-1}x+A_d.\vspace{.075cm}\]
If $d\in K^*$ and the following quantities are non-zero, 
\begin{enumerate}[topsep=8pt, partopsep=8pt, itemsep=13pt]   
\item[\textup{(1)}] $\delta_\phi:=2dA_0A_{2}-(d-1)A_{1}^2$,    
\item[\textup{(2)}] $\epsilon_\phi:=(d-1)^2A_0A_1A_1' + d(d-3)A_0A_2A_0' - (d-1)(d-2)A_1^2A_0'-d(d-1)A_0^2A_2'$,    
\end{enumerate} 
then the dynamical Zsigmondy set \vspace{.1cm} 
\[\mathcal{Z}(\phi,a,b)=\{n\,:\, \phi^n(b)-a\;\, \text{has no primitive prime divisors}\}  \vspace{.1cm} \] 
is finite for all $\phi$-wandering pairs $a,b\in K$.  
\end{theorem}
Likewise, we prove a similar statement for arbitrary function fields. However, any assertions about the finiteness of $\mathcal{Z}(\phi,a,b)$ are superseded by \cite{ABCimplies} in characteristic zero. Nonetheless, we find an explicit $(\ell,m)$ such that (\ref{curve}) is non-isotrivial, a useful fact for effectively bounding Zsigmondy sets.\vspace{.0075cm}                
\begin{theorem}{\label{thm:dyniso}} Let $K/k(t)$ be an arbitrary function field, let $\phi(x)\in K[x]$ have degree $d\geq5$ with $d\in K^*$, and write 
\[\phi(x)=A_0x^d+A_{1}x^{d-1}+\dots A_{d-1}x+A_d.\]
If $\phi$ and $a\in K$ satisfy the following conditions:     
\begin{enumerate}[topsep=8pt, partopsep=6pt, itemsep=10pt] 
\item[\textup{(1)}] $\phi^3(x)-a$ is irreducible over $K$,    
\item[\textup{(2)}] $\delta_\phi\neq0$, 
\item[\textup{(3)}] $\epsilon_\phi\neq0$.   
\end{enumerate} 
Then at least one of the hyperelliptic curves \vspace{.15cm}
\[\;C_{(a,2,1)}(\phi): Y^2=\phi(X)-a, \;\;\;\, C_{(a,2,2)}(\phi): Y^2=\phi^2(X)-a,\,\;\;\;\text{or}\;\;\;\, C_{(a,2,3)}(\phi): Y^2=\phi^3(X)-a\vspace{.15cm}\]
is non-isotrivial. In particular, $\mathcal{Z}(\phi,a,b)$ is finite for all $\phi$-wandering points $b\in K$.         
\end{theorem}
\begin{rmk} In fact, we prove the stronger statement that sequence $\phi^n(b)-a$ has a primitive prime divisors appearing to odd valuation for all $n$ sufficiently large (c.f. \cite[Theorem 1]{primdiv}).  
\end{rmk} 
Our motivation for studying prime divisors in orbits comes from the theory of dynamical Galois groups. Namely, for certain types of polynomials (e.g. unicritical polynomials and some trinomials), primitive prime divisors in critical orbits control the image size of arboreal representations; see Section \ref{sec:Galois} for the relevant definitions and \cite{Jonessurvey} for an introduction to the subject. In particular, we use Theorem \ref{thm:main} and some results of Looper for trinomials \cite{Looper} to construct finite index arboreal representations, $G_K(\phi)\leq\Aut(T_d)$, in every characteristic. 
\begin{theorem}{\label{thm:eg}} Let $K=k(t)$ be a rational function field of any characteristic, and let 
\[\phi_{p,B}(x)=x^p+\bigg(\frac{-pB^p-pB}{pB^{p-1}+p-1}\bigg)\,x^{p-1}+ B\]  
for some non-constant $B\in K(t)$ and some prime $p\geq3$. If $p(p-1)\in K^{*}$, then the image of the arboreal representation of $\phi_p$, $G_K(\phi_{p,B})\leq\Aut(T_p)$, is a finite index subgroup.  \vspace{.05cm}  
\end{theorem} \vspace{.1cm}
In 1985, Odoni conjectured that for each $d \geq 2$ there is a polynomial with integer coefficients whose arboreal representation has image as large as possible; see \cite[Conjecture 2.2]{Jonessurvey}). For isotrivial maps over function fields, it is known that Odoni's conjecture holds; see \cite[Theorem 3.1]{Juul}. However, the proof relies on special properties of these maps. Therefore, Theorem \ref{thm:eg} establishes a version of Odoni's conjecture (up to finite index) for non-isotrivial maps of prime degree defined over a function field. 

In the case of function fields of characteristic zero, we can prove stronger results, including surjectivity (i.e., index one) in some cases. There are several reasons for this improvement. First, isotriviality of the relevant hyperelliptic curves is not a problem in characteristic zero: effective height bounds are known for all hyperelliptic curves \cite{Mason,Schmidt}. Moreover, we have reduction maps $\mathbb{Z}[t]\rightarrow \mathbb{F}_p[t]$ for any prime $p$, and we can use these maps to reduce the complexity of the relevant dynamical factorization problems. In particular, we are able to prove that the arboreal representation of $\phi_{p,t}$ (i.e. $B=t$) is surjective by exploiting these properties. \vspace{.1cm}          
\begin{theorem}{\label{thm:Odoni}} Let $K=k(t)$ be a rational function field of characteristic zero, and let \vspace{.05cm} 
\[\phi_{p}(x)=x^p+\bigg(\frac{-pt^p-pt}{pt^{p-1}+p-1}\bigg)\,x^{p-1}+ t\vspace{.05cm}\]   
for some prime $p\geq3$. Then $G_K(\phi_{p})=\Aut(T_p)$, i.e., the arboreal representation of $\phi_p$ is surjective.    
\end{theorem}
Finally, although it is in general quite difficult to prove surjective (and finite index) results for arboreal representations, we can nonetheless show that ``most" quadratic polynomials furnish surjective representations over function fields in characteristic zero. To make a version of this statement precise, we fix some notation. For a polynomial $f(t)\in\mathbb{Z}[t]$, write 
\[ f(t)=a_dt^d+a_{d-1}t^{d-1}+\dots +a_0 \;\;\;\;\text{for}\; a_i\in\mathbb{Z}\] 
and define $h_{\cf}(f)=\max\{|a_i|\}$ to be the maximum absolute value of the coefficients of $f$. Moreover, for positive integers $d\geq1$ and $B\geq0$, let \vspace{.05cm} 
\[\mathcal{P}_d(B):=\{f\in\mathbb{Z}[t]:\,\deg(f)\leq d, \; h_{\cf}(f)\leq B\} \vspace{.05cm} \] 
be the the set of integral polynomials of degree at most $d$ and coefficients of absolute value at most $B$. It is clear that $\mathcal{P}_d(B)$ is a finite set. To every pair $\gamma(t), c(t)\in\mathbb{Z}[t]$, we can assign a quadratic polynomial over the rational function field $K=\mathbb{Q}(t)$, given by 
\[\phi_{(\gamma,c)}(x):=(x-\gamma)^2+c.\]
Now, consider the set  
\[\Sur_d(B):=\big\{(\gamma,c)\in\mathcal{P}_d(B)\times\mathcal{P}_d(B):\,G_{K}(\phi_{\gamma,c})=\Aut(T_2) \big\},\vspace{.05cm}\]
of all quadratic polynomials of this form having surjective arboreal representations over $K$. In particular, we show that for all \emph{fixed} $d\geq1$, the set $\Sur_d(B)$ approaches full asymptotic density in the set of all quadratic polynomials as $B$ tends to infinity. The main tool that we use to establish this result is the uniform finite index theorem for quadratic polynomials established by the first author in \cite{uniform}.     
\begin{theorem}{\label{thm:count}} For all $d\geq1$ and $B\geq0$, 
\[\#\Sur_d(B)=(2B+1)^{(2d+2)}+O(B^{2d+2-1/2}\log B).\]
In particular, 
\[\lim_{B\rightarrow\infty}\;\frac{\#\Sur_d(B)}{\#\,\mathcal{P}_d(B)\times\mathcal{P}_d(B)}=1.\vspace{.1cm}\] 
Hence, almost all quadratic polynomials of the form $\phi_{(\gamma,c)}(x):=(x-\gamma)^2+c$, given by $\gamma(t),c(t)\in\mathbb{Z}[t]$ and $\deg(\gamma),\deg(c)\leq d$, have surjective arboreal representations over $\mathbb{Q}(t)$. 
\end{theorem} 
\vspace{.25cm}
\indent \textbf{Acknowledgements:} It is a pleasure to thank Felipe Voloch for several useful conversations throughout this work, especially those related to Lemma \ref{lem:RicattitoIso}.  
\section{Riccati equations in orbits and dynamical Zsigmondy sets}
In this section, we prove an isotriviality test for the curves in (\ref{curve}) involving a differential equation on the roots of its defining polynomial. This type of argument was developed recently by the first author to prove a version of Silverman's dynamical integral point theorem \cite{Silv-Int} over function fields; see \cite[Theorem 1.1]{IntOrb}. In particular, our goal is to show that iterated preimages (of most basepoints under most polynomials) eventually avoid Riccati equations, which is a consequence of the following fundamental lemma (c.f. \cite[Lemma 2.3]{IntOrb}). In what follows, $K^{\sep}$ denotes the separable closure of $K$. 
\begin{lemma}{\label{lem:unique}} Let $K/k(t)$, let $\phi(x)\in K[x]$ have degree $d\geq3$, and write \vspace{.025cm} 
\[\phi(x)=A_0x^d+A_{1}x^{d-1}+\dots A_{d-1}x+A_d.\vspace{.075cm}\]    
If $d\in K^*$ and the quantity \[\delta_\phi:=2dA_0A_{2}-(d-1)A_{1}^2\vspace{.15cm}\]
is non-zero, then for all $\beta\in K^{\sep}$ such that $\beta$ and $\phi(\beta)$ both satisfy a Riccati equation, i.e.\vspace{.1cm} 
\begin{equation}{\label{Ricatti}} \beta'=a\beta^2+b\beta+c\;\;\;\text{and}\;\;\;\phi(\beta)'=e\phi(\beta)^2+f\phi(\beta)+g\vspace{.1cm}
\end{equation} 
for some $a,b,c,e,f,g\in K$, either \[[K(\beta):K]\leq2d,\] or the coefficients in (\ref{Ricatti}) are uniquely determined by $\phi$:  \vspace{.15cm} 
\begin{equation}{\label{solutions}}
\begin{split} 
&\,a=0,\;\;\;\; b=(dA_0^2A_2' - (d-1)A_0A_1A_1' - dA_0A_2A_0' + (d-1)A_1^2A_0')/\delta_\phi ,\\[8pt] 
&\,e=0,\;\;\; f=(d^2A_0^2A_2' - d(d-1)A_0A_1A_1' - d(d-2)A_0A_2A_0' + (d(d-2)+1)A_1^2A_0')/\delta_\phi, \\[8pt] 
&\,c=(A_0A_1A_2' - 2A_0A_2A_1' + A_1A_2A_0')/\delta_\phi, \;\;\; g=A_{d-1}c-A_df+A_d'. \\[3pt]  
\end{split}    
\end{equation} 
\end{lemma} 
\begin{proof} Assume that $\beta\in K^{\sep}$ is such that $\beta$ and $\phi(\beta)$ both satisfy a Riccati equation. Then, after differentiating the expression $A_0\beta^d+A_1\beta^{d-1}+\dots A_{d-1}\beta+ A_d=\phi(\beta)$, we obtain: \vspace{.1cm} 
\begin{equation}{\label{deriv}}
F_1(\beta)\beta'+F_2(\beta)=\phi(\beta)'. \vspace{.1cm}
\end{equation} 
Here the polynomials $F_1$ and $F_2$, associated to $\phi$, are given by \vspace{.05cm}
\[F_1(x)=dA_0x^{d-1}+(d-1)A_{1}x^{d-2}+\dots+ A_{d-1}\;\;\;\text{and}\;\;\;\; F_2(x)=A_0'x^d+\dots+A_{d-1}'x+ A_d'. \vspace{.05cm}\]
In particular, we substitute (\ref{Ricatti}) into (\ref{deriv}) and compute that \vspace{.05cm}
\begin{equation*}
F_1(\beta)(a\beta^2+b\beta+c)+F_2(\beta)=e{\phi(\beta)}^2+f\phi(\beta)+g. 
\end{equation*}
Hence, we obtain a polynomial $P_{\phi,\beta}\in K[x]$ satisfying 
\[\deg(P_{\phi,\beta})\leq2d\;\;\;\text{and}\;\;\;P_{\phi,\beta}(\beta)=0.\]
Therefore, if $[K(\beta):K]>2d$, then $P_{\phi,\beta}$ must be the zero polynomial. In particular, since 
\[P_{\phi,\beta}(x)=e(A_0x^d+O(x^{d-1}))^2-(dA_0)ax^{d+1}+O(x^d),\]
we see immediately that $e=0$ and $a=0$; here we use that $2d>d+1$ and $d\in K^*$. Consequently, \vspace{.1cm}   
\begin{equation}\label{poly}
\begin{split} 
P_{\phi,\beta}(x)=&\big(dA_0b-A_0f+A_0'\big)\,x^d+\big((d-1)A_1b-A_1f+dA_0c+A_1'\big)\,x^{d-1}\;\;+\\[5pt]
&\big((d-2)A_2b-A_2f+(d-1)A_1c+A_2'\big)\,x^{d-2}+O(x^{d-3}).
\end{split} 
\end{equation} 
Therefore, the vector $(b,f,c)\in K^3$ is a solution to the linear system of equations:  
\begin{equation}{\label{matrix}}
\begin{pmatrix}
dA_0&-A_0&0\\
(d-1)A_1&-A_1&dA_0\\
(d-2)A_2&-A_2&(d-1)A_1 
\end{pmatrix}
\begin{pmatrix}
\textbf{x}_1\\
\textbf{x}_2\\
\textbf{x}_3 
\end{pmatrix}
=
\begin{pmatrix}
A_0'\\
A_1'\\
A_2' 
\end{pmatrix}
\end{equation}
Moreover, since the determinant $\delta_\phi=2dA_0A_{2}-(d-1)A_{1}^2$ of the coefficient matrix in (\ref{matrix}) is non-zero, the Riccati coefficients $a,b,c,e,f,g\in K$ in (\ref{Ricatti}) are uniquely determined by $\phi$; note that the constant term of $P_{\phi,\beta}$ is $A_{d-1}c-A_df+A_d'-g=0$, and so $g$ is determined by $c$ and $f$. In particular, we obtain the description of the Riccati coefficients $(a,b,c,e,f,g)$ in (\ref{solutions}) by solving the linear system in (\ref{matrix}).                          
\end{proof} 
We now relate Riccati equations to the curves in (\ref{curve}).  
\begin{lemma}{\label{lem:RicattitoIso}} Let $K/k(t)$ and suppose that $\rho(x)\in K[x]$ is an irreducible (and separable) polynomial of degree $d\geq5$. If $\beta\in K^{\sep}$ is such that $\rho(\beta)=0$ and $\beta$ does not satisfy a Riccati equation over $K$, i.e. 
\[\beta'\neq a\beta^2+b\beta+c\]
for all $a,b,c\in K$, then the hyperelliptic curve\, $C: Y^2=\rho(X)$ is non-isotrivial.   
\end{lemma} 
\begin{proof} Consider the affine curve \[Y:=\mathbb{P}^1\mysetminus\big\{\beta\in K^{\sep}\,:\, \rho(\beta)=0\big\}.\]
We show first that $Y$ is non-isotrivial (whenever $\beta$ is non-Riccati) and then show that this implies that $C$ is non-isotrivial. Suppose for a contradiction that $Y$ is isotrivial. Then $Y$ is isomorphic to an affine curve $Y'$ defined over $\overline{k}$. Hence, $Y'$ is an open subset of $\mathbb{P}^1$, and its complement is a finite set of points defined over $\overline{k}$. Moreover, the corresponding map $Y\rightarrow Y'$ must be a linear fractional transformation. In particular, since such maps preserve cross ratios, we see that the cross ratio of any four conjugates of $\beta\in K^{\sep}$ with $\rho(\beta)=0$ must be in a finite extension of the constant field. Consequently, since the derivative of any element of $\overline{k}$ is zero, a straightforward calculation \cite[\S7 Claim 1]{Voloch} shows that $\beta$ must satisfy a Riccati equation over $K$, a contradiction. Therefore, $Y$ is non-isotrivial.      

Because $d\geq5$, and hence the genus of $C$ is at least $2$, the hyperelliptic map $x:C\rightarrow\mathbb{P}^1$ is unique up to automorphisms of $\mathbb{P}^1$; see, for instance, \cite[IV Proposition 5.3]{Hartshorne}. In particular, if $C$ is isotrivial, i.e. isomorphic (over $\overline{K}$) to a curve $C'$ defined over an extension $k'/k$ of the constant field, then $C': v^2=f(u)$ for some $f\in k'[u]$ and the isomorphism $\theta:C\rightarrow C'$ must fit in a commutative diagram,
\begin{displaymath}
    \Scale[1.2]{\xymatrix{
        	 C\ar[r]^{\theta} \ar[d]^{x\;\;\;\;} & C' \ar[d]^{u\;\;\;\;} \\
        \mathbb{P}^1 \ar[r]^{\lambda\;} & \mathbb{P}^1 }}
\end{displaymath} 
for some linear fractional transformation $\lambda\in\PGL_2(\overline{K})$. Consequently, $\lambda$ must take branch points of $x:C\rightarrow\mathbb{P}^1$ to branch points of $u:C'\rightarrow\mathbb{P}^1$. Moreover, we can assume that $C'$ has an even degree model (i.e. $\deg(f)$ is even) by passing to some finite extension $k^{''}/k'/k$ if need be: if $\deg(f)$ is odd, simply move an affine Weierstrass point (perhaps defined over an extension of $k'$) to infinity with a linear fractional transformation on the $u$-coordinate; see, for instance, \cite[Proposition 2.1]{evenmodel}. Therefore, 
\[\lambda\big(\big\{\beta\in K^{\sep}\,: \rho(\beta)=0\big\}\big)\subseteq \big\{\alpha\in \overline{k}\,: f(\alpha)=0\big\},\]
since $u$ has only affine branch points. In particular, $\lambda$ restricts to an isomorphism 
\[Y\longrightarrow\mathbb{P}^1\mysetminus\big\{\alpha\in \overline{k}\,: f(\alpha)=0 \;\text{and}\;\alpha=\lambda(\beta)\;\text{for some}\;\rho(\beta)=0\big\}.\]
Hence $Y$ must be isotrivial, a contradiction. Therefore, $C$ is non-isotrivial as claimed.         
\end{proof}    
\begin{proof}[(Proof of Theorem \ref{thm:main})] The first step in our argument is to show that iterated preimages of $a$ are eventually non-Riccati, i.e. $\beta\in\phi^{-m}(a)$ implies that $\beta$ does not satisfy a Riccati equation (over $K$) for $m$ sufficiently large. To see this, we fix some notation. Let \vspace{.1cm}
\begin{equation}{\label{htmin}} \hat{h}_{\phi,K}^{\min}(2d):=\inf\Big\{\,\hat{h}_\phi(\alpha)\,:\,\alpha\in\mathbb{P}^1(\overline{K}),\; [K(\alpha):K]\leq2d,\;\hat{h}_\phi(\alpha)>0\Big\}
\end{equation} 
be the minimum non-zero canonical height of points of degree at most $2d$ over $K$. Note in particular that $\hat{h}_{\phi,K}(2d)$ is positive. To see this, choose an arbitrary point $c_0\in K$ such that $\hat{h}_\phi(c_0)>0$, possible by Northcott's Theorem \cite[Theorem 3.7]{Silv-Dyn} for global function fields; see, for instance, \cite[\S3.3]{Lang2}. Now note that 
\[\hat{h}^{\min}_{\phi,K}(2d)=\inf\Big\{\,\hat{h}_\phi(\alpha)\,:\,\alpha\in\mathbb{P}^1(\overline{K}),\; [K(\alpha):K]\leq2d,\;0<\hat{h}_\phi(\alpha)<\hat{h}_\phi(c_0)\Big\}.\] However, this latter set is finite (again by Northcott's Theorem) and consists of strictly positive numbers. Hence $\hat{h}^{\min}_{\phi,K}(2d)$ is positive, as claimed. 

Now, if $\beta\in\overline{K}$ is such that $\phi^m(\beta)=a$ for some iterate $m$ satisfying
\begin{equation}{\label{lbdn}} 
\log_d\bigg(\frac{\hat{h}_\phi(a)}{\hat{h}^{\min}_{\phi,K}(2d)}\bigg)+3\leq m,
\end{equation}  
then $\beta$, $\phi(\beta)$ and $\phi^2(\beta)$ all have degree strictly larger than $2d$ over $K$. For instance, if $\phi^m(\beta)=a$ and $[K(\phi^i(\beta)):K]\leq 2d$ for some $0\leq i\leq2$, then \vspace{.1cm} 
\[d^{m-i}\cdot \hat{h}^{\min}_{\phi,K}(2d)\leq d^{m-i}\cdot\hat{h}_\phi(\phi^i(\beta))=\hat{h}_\phi(\phi^{m-i}(\phi^i(\beta)))= \hat{h}_\phi(\phi^{m}(\beta))=\hat{h}_\phi(a).\vspace{.05cm} \]
Hence, $m-i\leq \log_d\big(\,\hat{h}_\phi(a)\big/\hat{h}^{\min}_{\phi,K}(2d)\,\big)$, a contradiction of the lower bound in (\ref{lbdn}). Therefore,
\begin{equation}{\label{degreelbd}} 
[K(\phi^i(\beta)):K]>2d\vspace{.05cm}
\end{equation}  
for all $0\leq i\leq2$ and all $m$ as in (\ref{lbdn}); note here that we have used crucially that $\hat{h}_\phi(a)>0$, from which it follows that $\hat{h}_\phi(\beta)>0$ also.  

On the other hand, we claim that at least one of the elements of $\{\beta, \phi(\beta), \phi^2(\beta)\}$ does not satisfy a Riccati equation over $K$. To see this, first choose a separable preimage $\beta\in \phi^{-m}(a)$: if every root of $\phi^m(x)-a$ is inseparable, then $\deg(\phi^m(x)-a)=d^n$ is divisible by $\car(K)$ \cite[\S13.5 Proposition 38]{DF}. However, this  contradicts our assumption that $d\in K^*$. On the other hand, if $\beta\in K^{\sep}$ and $\{\beta, \phi(\beta), \phi^2(\beta)\}$ all satisfy a Riccati equation over $K$, then Lemma \ref{lem:unique} and (\ref{degreelbd}) applied to the pairs $\big(\beta,\phi(\beta)\big)$ and $\big(\phi(\beta),\phi^2(\beta)\big)$ imply that \vspace{.05cm}
\begin{equation} 
\begin{split} 
\beta'=b\beta+c\;\;\;\;&\text{and}\;\;\;\;\phi(\beta)'=f\phi(\beta)+g,\\[3pt]
\phi(\beta)'=b\phi(\beta)+c\;\;\;\;&\text{and}\;\;\;\;\phi^2(\beta)'=f\phi^2(\beta)+g\\[3pt] 
\end{split} 
\end{equation}  
for $(b,c,f,g)\in K^4$ as in (\ref{solutions}). In particular, we see that $(b-f)\phi(\beta)+(c-g)=0$. Therefore, both $(b-f)=0$ and $(c-g)=0$, since $\phi(\beta)\not\in K$ by (\ref{degreelbd}). Consequently, \vspace{.1cm}  
\[\epsilon_\phi:=(b-f)\delta_\phi=(d-1)^2A_0A_1A_1' + d(d-3)A_0A_2A_0' - (d-1)(d-2)A_1^2A_0'-d(d-1)A_0^2A_2' \vspace{.1cm} \]
must vanish. However, this fact contradicts assumption (2) of Theorem \ref{thm:main}. Hence, we can choose $m$ sufficiently large such that there exists a separable $\beta\in\phi^{-m}(a)$ such that at least one of the elements of $\{\beta, \phi(\beta), \phi^2(\beta)\}$ is non-Riccati. Without loss of generality, we may assume that $\beta$ is non-Riccati (after possibly replacing $\phi^i(\beta)$ with $\beta$), and we fix such an $m$ and $\beta$ once and for all.  

From here, we prove the finiteness of Zsigmondy sets using (essentially) the same argument given for \cite[Theorem 1.1]{Trans} and \cite[Theorem 1]{primdiv}, which we now sketch. The reader should keep in mind that the full (correct) argument requires some necessary alterations, but the main idea is the following: if $\phi^n(b)-a$ does not have primitive prime divisors appearing to odd valuation for some $n>m$, then the ``square-free" part $d_n$ of $\phi^n(b)-a$ is small. In particular, the primes of bad reduction of the hyperelliptic curve $C: d_n Y^2=\phi^m(X)-a$ are also small ($m$ is fixed). Therefore, effective versions of the Mordell conjecture (a theorem in this setting) for \emph{non-isotrivial curves} over function fields implies that any rational point on $C$ has small height \cite{Kim,htineq}; here we use Lemma \ref{lem:RicattitoIso}. But $X=\phi^{n-m}(b)$ gives a rational point on $C$, from which we obtain a bound on $n$. To make this argument precise, we fix some notation. Write:   \vspace{.1cm}  
\[\phi^m(x)-a=f_1(x)^{e_1}\,f_2(x)^{e_2}\dots\, f_t(x)^{e_t}, \; \text{where} \vspace{.1cm}\] 
\hspace{3cm} 
\begin{varwidth}{\textwidth}
\begin{enumerate}{\label{reductions}}
\item \;The $f_i$ are distinct and irreducible over $K$,\vspace{.1cm}
\item \;The roots of $f_1$, including $\beta$, are non-Riccati, \vspace{.1cm}
\item \; $\deg(f_i)\geq2d\geq6$.\vspace{.5cm}
\end{enumerate}  
\end{varwidth} 
Moreover, let $R_j:=\Res(f_1,f_j)$ for $1<j\leq t$, and choose a finite set $S\subseteq V_K$ so that:\vspace{.15cm}  
\begin{align*}
 (i)&\;\; a,b\in\mathcal{O}_{K, S},\hspace{1.6cm}(ii)\;\; \text{$\phi$ and each $f_i$ has good reduction outside $S$,}\\[4pt] 
(iii)&\;\; \text{$\mathcal{O}_{K, S}$ is a UFD,} \hspace{.92cm}(iv)\;\;\{v\,: v(R_j)>0,\; 1<j\leq t\}\subseteq S. 
\end{align*}
Note that we can enlarge any finite set of places $S$ to satisfy conditions (i)-(iv); see, for instance, \cite[Proposition 14.2]{Rosen}. Now, since $\mathcal{O}_{K,S}$ is a UFD, we can decompose any term of the form $f_1(\phi^n(b))\in\mathcal{O}_{K,S}$ for $n\geq1$ in the following way: \vspace{.1cm}  
\begin{lemma}{\label{lem:decomp}} Let $\phi$, $K$ and $S$ be as above. Then we have a decomposition 
\begin{equation*}{\label{decomp}} f_1(\phi^n(b))=u_n\cdot d_n\cdot y_n^2,\,\;\;\;\text{for some}\;\; d_n,y_n\in\mathcal{O}_{K,S},\;u_n\in\mathcal{O}_{K,S}^*,
\end{equation*} 
satisfying the following properties:
\begin{enumerate}[topsep=3mm,itemsep=3mm]
\item[\textup{(a)}] $0\leq v(d_n)\leq1$ for all $v\not\in S$, 
\item[\textup{(b)}] There is a constant $r(S)$ such that $0\leq v(d_n)\leq r(S)$ for all $v\in S$.
\item[\textup{(c)}] The set $\{u_n\}_{n\geq0}$ is finite.    
\end{enumerate}   
In particular, $d_n\in\mathcal{O}_K$ and the height of $u_n$ is bounded independently of $n$.  
\end{lemma} 
For a proof of this fact, see \cite[Lemma 2.2]{Trans}. We now show that for all $n$ sufficiently large, there exists $v_n\in V_K\mysetminus S$ such that: \vspace{.1cm}
\begin{equation}{\label{eq:Zsig}}
v_n(f_1(\phi^n(b)))\equiv 1\Mod{2}\;\;\;\; \text{and}\;\;\;\; v_n(\phi^{i}(b)-a)=0\;\;\text{for all}\; 1\leq i\leq m+n-1. \vspace{.1cm} 
\end{equation} 
Moreover, if (\ref{eq:Zsig}) holds, then it follows from condition (iv) that
\begin{equation}{\label{valuation}}
v_n(\phi^{n+m}(b)-a)=e_1\, v_n(f_1(\phi^n(b)))+\sum_{j=2}^te_j\, v_n(f_j(\phi^n(b)))=e_1\, v_n(f_1(\phi^n(b))).
\end{equation} 
Hence, $\phi^{n+m}(b)-a$ has a primitive prime divisor for all $n$ sufficiently large. Moreover, if in addition $e_1$ is odd, then $\phi^{n+m}(b)-a$ has a primitive prime divisors appearing to odd valuation.   

Suppose, for a contradiction, that (\ref{eq:Zsig}) does not hold. Then Lemma \ref{lem:decomp} implies that the support of $d_n$ (away from $S$) is contained within the union of the support of lower order iterates:\vspace{.1cm} 
\[\big\{v\in V_K\mysetminus{S}\,:\, v(d_n)>0\big\}\subseteq\big\{v\in V_K\mysetminus S\,:\, v(\phi^j(b)-a)>0\;\text{for some}\; 1\leq j\leq m+n-1\big\}. \vspace{.1cm} \]
Therefore, since $d_n\in\mathcal{O}_K$ is an integer and $\phi$ has good reduction outside of $S$, Lemma \ref{lem:decomp} and \cite[Lemma 2.4]{Trans} imply that that there is a constant $c(K,S)$ such that  
\begin{equation}{\label{htestimate}}h(d_n)\leq \Bigg(\sum_{i=1}^{\lfloor\frac{n+m}{2}\rfloor}h(\phi^i(b)-a)+ \sum_{j=1}^{\lfloor\frac{n+m}{2}\rfloor}h(\phi^j(a)-a)\Bigg)+c(K,S),    
\end{equation} 
where $h:\mathbb{P}^1(\overline{K})\rightarrow\mathbb{R}_{\geq0}$ is the standard Weil height \cite[\S2]{Trans}. Here we use that if $v\in V_K\mysetminus S$ is such that $v(d_n)>0$ and $v(\phi^{j}(b)-a)>0$ for some $1\leq j\leq n+m-1$, then $v(\phi^{m+n}(b)-a)>0$ and $v(\phi^{m+n-j}(a)-a)>0$. Moreover, note that when $a=0$ and $(\ell,m)=(2,0)$, the height estimate in (\ref{htestimate}) is identical to the estimate in \cite[\S2 (12)]{Trans}.   

Now we use the fact that some element of the backwards orbit of $a$ is non-Riccati. In particular, property (2) of the decomposition of $\phi^m(x)-a$ and Lemma \ref{lem:RicattitoIso} implies that the curve $C:Y^2=f_1(X)$ is non-isotrivial. Therefore, we can use any of the effective versions of the established Mordell conjecture \cite{Kim,htineq} to bound the height of algebraic points on $C$; see also \cite[Remark 2.10]{Trans}. Specifically, since 
\[\Big(\phi^{n}(b),\sqrt{u_n\,d_n\,}\,y_n\Big)\in C(\overline{K})\] 
is an algebraic point on $C$, there exist constants $B_1$ and $B_2$ (independent of $n$) so that \vspace{.1cm} 
\begin{equation}{\label{upbd}}
\begin{split}  
h(\phi^n(b))&\leq B_1\,h(d_n)+B_2\\[3pt]
&\leq B_1\Bigg(\sum_{i=1}^{\lfloor\frac{n+m}{2}\rfloor}h(\phi^i(b)-a)+ \sum_{j=1}^{\lfloor\frac{n+m}{2}\rfloor}h(\phi^j(a)-a)\Bigg)+B_1\cdot c(K,S)+B_2.  
\end{split} 
\end{equation}
On the other hand, the Weil height and canonical height are comparable functions on $\mathbb{P}^1(\overline{K})$, that is, $h=\hat{h}_\phi +O(1)$, and $\hat{h}_\phi(\phi^s(P))=d^s\hat{h}_\phi(P)$ for all $P\in\mathbb{P}^1(\overline{K})$; see \cite[Theorem 3.20]{Silv-Dyn}. Moreover, 
\[h(P-a)\leq h(P)+h(a)+\log(2)\;\;\;\;\text{for all}\;P\in \mathbb{P}^1(\overline{K}),\] 
by the triangle inequality applied at every place of $K$; see \cite[Exercise 8.8]{Silv-Ell}. In particular, since $\hat{h}_\phi(b)\neq0$, it follows from (\ref{upbd}) that 
\begin{equation}{\label{bd}} d^{n}\leq B_3\,d^{\lfloor\frac{n+m}{2}\rfloor+1}+B_4\,n+ B_5,
\end{equation} 
for some constants $B_3$, $B_4$, and $B_5$ (depending on $a$, but not on $n$). Finally, since $d>1$ and $m$ is constant, the bound in (\ref{bd}) implies that $n$ is bounded. Hence $\mathcal{Z}(\phi,a,b)$ is finite, as claimed.    
\end{proof}   
\begin{proof}[(Proof of Theorem \ref{thm:dyniso})] The argument here is only a slight modification of that given for Theorem \ref{thm:main}. The key difference is that Northcott's theorem fails for characteristic zero function fields. In particular, $\hat{h}_{\phi,K}(2d)$, defined in (\ref{htmin}), need not be positive in this case, and our proof that there exists non-Riccati $\beta\in\phi^{-m}(a)$ for $m$ sufficiently large breaks down. On the other hand, if $\phi^3(x)-a$ is irreducible over $K$, then so are $\phi^2(x)-a$ and $\phi(x)-a$. Therefore, 
\[2d<[K(\beta):K]=d^3\;\;\;\;\text{and}\;\;\;\;\;2d<[K(\phi(\beta)):K]=d^2\] 
for all $\beta\in\phi^{-3}(a)$. In particular, since $\delta_\phi$ and $\epsilon_\phi$ are non-vanishing, Lemma \ref{lem:unique} implies that at least one of the elements  in $\{\beta,\phi(\beta),\phi^2(\beta)\}$ does not satisfy a Riccati equation over $K$, as we argued in the proof of Theorem \ref{thm:main}. Hence, at least one of the hyperelliptic curves 
\[\;C_{(a,2,1)}(\phi): Y^2=\phi(X)-a, \;\;\;\, C_{(a,2,2)}(\phi): Y^2=\phi^2(X)-a,\,\;\;\;\text{or}\;\;\;\, C_{(a,2,3)}(\phi): Y^2=\phi^3(X)-a\vspace{.15cm}\]
is non-isotrivial by Lemma \ref{lem:RicattitoIso} (here we use also that $d\geq5$). Now we repeat the proof of Theorem \ref{thm:main} by setting $m=1,2$ or $3$ and $f_1(x)=\phi^m(x)$ in the decomposition of $\phi^m(x)-a$ above. In particular, we deduce that $\mathcal{Z}(\phi,a,b)$ is finite for all $\phi$-wandering $b\in K$. Moreover, (\ref{eq:Zsig}) implies the stronger statement: $\phi^n(b)-a$ has a primitive prime divisor with odd valuation for all $n$ sufficiently large (compare to \cite[Theorem 1]{primdiv}). 
\end{proof}

\section{Ramification theory and group theory}{\label{sec:Galois}} 
In this section, we generalize several of the results in \cite{Looper} on the Galois groups of iterates of trinomials (themselves generalizations of results in \cite{Cohen, ramtri}) to Dedekind domains with perfect residue fields. However, we first recall the basic notions of arboreal representations attached to rational maps. Let $K$ be a field and let $\phi(x)\in K(x)$ be a rational function of degree $d\geq2$.

For $n\geq1$, let $K_n=K_n(\phi)$ be the field extension of $K$ obtained by adjoining all solutions in $\overline{K}$ of $\phi^n(x)=0$ to $K$. Generically, the extension $K_n/K$ is Galois, and we define $G_{K,n}(\phi)$ to be the Galois group of $K_n(\phi)$ over $K$. Since $K_{n-1}(\phi)\subseteq K_{n}(\phi)$ for all $n\geq1$ with some mild separability assumptions, we may define   
\begin{equation*}{\label{Galois}} G_K(\phi)=\lim_{\longleftarrow}G_{K,n}(\phi)
\end{equation*}
with respect to the restriction maps. Dynamical analogs on $\mathbb{P}^1$ of the Galois representations attached to abelian varieties \cite{Serre-reps} (where one instead considers iterated preimages of multiplication maps), the groups $G_K(\phi)$ have obtained much attention in recent years. 

Of course, a key difference in this setting is the lack of group structure on projective space, and as such $G_K(\phi)$ may often only be viewed as a subgroup of the automorphism group of a tree (or a wreath product) and not inside a group of matrices. Explicitly, let $T_d$ denote the regular infinite $d$-ary rooted tree. If we write $\phi^n(x)=f_n(x)/g_n(x)$ for some $f_n,g_n\in K[x]$ such that $\disc({f_n})\neq0$ for all $n\geq1$, then we may identify the vertices of $T_d$ with the set of iterated preimages of zero (under $\phi$) and define an edge relation on this set by: $\alpha,\beta\in T_d$ share an edge if and only if $\phi(\alpha) =\beta$ or $\phi(\beta)=\alpha$. In particular, since Galois commutes with polynomial evaluation, we have an inclusion 
\[G_K(\phi)\leq\Aut(T_d)\] 
called the \emph{image of the arboreal representation} (or just \emph{arboreal representation} when no confusion is possible) associated to $\phi$. A major goal of dynamical Galois theory is to understand the subgroup $G_K(\phi)$ of $\Aut(T_d)$; see \cite{B-J} and \cite{Jonessurvey} for detailed introductions to the subject.  

To begin our study of Galois groups of iterated trinomials, let $d \geq 2$ be an integer, and let the symmetric group $S_d$ have its natural action on $\{1, \ldots, d\}$. A fundamental group-theoretic fact that we use is the following: 
\begin{proposition} \label{suffconds}
Let $H$ be a subgroup of $S_d$. The following imply that $H = S_d$:
\begin{enumerate}
\item[\textup{(1)}] $H$ is normal in $S_d$, and $H$ contains a transposition.
\item[\textup{(2)}] $d=p$ is a prime and $H$ is a transitive subgroup that contains a transposition.
\end{enumerate}
\end{proposition}
\begin{proof} Suppose that $H$ is a normal subgroup of $S_d$ that contains a transposition. Since the conjugacy classes of $S_d$ are determined by cycle type \cite[\S4.3 Proposition 11]{DF}, $H$ contains all transpositions. Therefore, $H$ is the full symmetric group: it is well known that $S_d$ is generated by transpositions \cite[Lemma 4.4.2]{Serre-Galois}. On the other hand, part (2) is a classic result of Jordan \cite{Jordan}.     
\end{proof} 
We also note the following useful observation on the reduction of roots of polynomials. 
\begin{proposition} \label{simple}
Let $S$ be a Dedekind domain, $\alpha_1, \ldots, \alpha_k$ elements of $S$ and $\q \subset S$ a prime ideal. Let $f(x) = \prod_{i = 1}^k (x - \alpha_i) \in S[x]$ and denote by $\bar{f} \in (S/ \q)[x]$ the polynomial obtained by reducing the coefficients of $f$ modulo $\q$. If $\bar{f}$ has $k$ distinct roots, then $\alpha_i \not\equiv \alpha_j \bmod{\q}$ for all $i \neq j$. 
\end{proposition}

\begin{proof}
Let $\pi : S \to S /\q$ be the natural homomorphism. Because $\pi$ is a ring homomorphism, it takes roots of $f$ to roots of $\bar{f}$. Thus the elements $\pi(\alpha_1), \ldots, \pi(\alpha_k)$ furnish $k$ roots (not a priori distinct) of $\bar{f}$. But $\bar{f}$ has at most $k$ roots, counting multiplicity, and hence every root of $\bar{f}$ is of the form $\pi(\alpha_i)$ for some $i$. By assumption $\bar{f}$ has $k$ distinct roots, and hence $\pi(\alpha_i) \neq \pi(\alpha_j)$ for all $i \neq j$, proving the proposition.
\end{proof}

We now prove a version of \cite[Theorem 2.1]{Looper} that holds for general Dedekind domains with perfect residue fields.
\begin{theorem}{\label{thm:transposition}}
Let $R$ be a Dedekind domain with quotient field $K$, let $f(x) = x^d + Ax^{s} + B \in R[x]$ with $d > s \geq 1$, let $L$ be the splitting field of $f$ over $K$, and let $S$ be the integral closure of $R$ in $L$. Assume that $ds(d-s) \neq 0$ in $R$, $\gcd(d,s) = 1$ and $f(x)$ is irreducible over $K$. If $\p$ is a prime of $R$ ramifying in $S$ and not dividing $AB$, then for any prime $\q$ of $S$ lying over $K$, the inertia group $I(\q /\p)$ has order two, and the non-trivial element acts on the roots of $f$ as a transposition. 
\end{theorem}
\begin{proof} We begin with a few observations. Note that the derivative $f'(x) = x^{s-1}(dx^{d-s} + sA)$, and so $d\in K^{*}$ implies that $f$ is separable (\cite[\S13.5 Proposition 33]{DF}); hence, $L/K$ is Galois. Moreover, letting $\p$ and $\q$ be as in the theorem, we note that the roots of $f$ are contained in $S$, and hence the reduced polynomial $\bar{f} \in (R/\p R)[x]$ splits completely in $S/\q S$.  

Furthermore, the discriminant of $f$ is given by the following formula \cite[Lemma 4]{Smith}:  
\begin{equation} \label{disc}
\delta = (-1)^{d(d-1)/2}B^{s-1}(d^dB^{d-s} + (-1)^{d-1}(d-s)^{d-s}s^sA^d)
\end{equation}
Moreover, since $\p$ is assumed to ramify in $S$, we have $\p \mid \delta$; see \cite[III Corollary 2.12]{Neu} combined with \cite[III \S3]{Lang}. Therefore, $\p \nmid ds(d-s)$, for otherwise $\p \mid \delta$, $\p \nmid AB$, and $\gcd(d,s) = 1$ give a contradiction. 

Note that $f'(x)$ has degree $d-1$, since $d\in K^{*}$. Let $\bar{f'} \in (R/ \p R)[x]$ be the polynomial obtained by reducing the coefficients of $f'$ modulo $\p$. Because $\p \nmid ds(d-s)$, the degree of $\bar{f'}$ is $d-1$ and the polynomial $dx^{d-s} + sA$ is separable over $K$. Thus the only irreducible factor of $\bar{f'}$ in $(R/ \p R)[x]$ that has multiplicity greater than one is $x$. In particular, since a root of order $r$ of $\bar{f}$ is a root of order at least $r-1$ of $\bar{f'}$, it follows that every root of $\bar{f}$ in $S/\q S$ has order at most $2$ (note that $0$ is not a root of $\bar{f}$ since $\p$ does not divide $B$).   

Now, since $\p \mid \delta$, the discriminant of $\bar{f}$ must vanish (taking the discriminant commutes with reduction). Therefore $\bar{f}$ must have at least one double root. Suppose that there are two such double roots $\beta_1, \beta_2 \in S/\q S$. Then $\beta_1, \beta_2 \neq 0$ and $\bar{f'}(\beta_1) = \bar{f'}(\beta_2) = 0$, implying that $\beta_1^{d-s} = -sA/d = \beta_2^{d-s}$. On the other hand, $\bar{f}(\beta_1) = \bar{f}(\beta_2) = 0$, and hence:  \vspace{.05cm}
\[\beta_1^s = - A -B/\beta_1^{d-s} = - A -B/\beta_2^{d-s} = \beta_2^s. \vspace{.05cm}\] 
Therefore, $(\beta_2/\beta_1)^{d-s} = (\beta_2/\beta_1)^s = 1$. However, $\gcd((d-s),s) = 1$, and we conclude that $\beta_2 = \beta_1$. Therefore $\bar{f}$ has a unique double root $\beta_1 \in S/\q S$, and hence $\bar{f}/(x-\beta_1)$ has degree $d-1$ and has $d-1$ distinct roots in $S/\q S$. Write 
\[ f(x) = \prod_{i=1}^{d} (x - \alpha_i) \in S[x],\] 
where $\alpha_1, \alpha_2$ are the two roots reducing to $\beta_1$ modulo $\q$. We now apply Proposition \ref{simple} to $f(x)/(x-\alpha_1)$ and obtain that $\alpha_i \not\equiv \alpha_j \bmod{\q}$ for all $i,j \in \{2, \ldots, d\}$ with $i \neq j$. Thus $I(\q/\p)$ acts trivially on $\{\alpha_3, \ldots, \alpha_{d}\}$. But $\p$ ramifies in $S$, so $I(\q/\p)$ cannot be trivial, and hence must have a single non-trivial element that interchanges $\alpha_1$ and $\alpha_2$.
\end{proof}

We now prove a dynamical version of Theorem \ref{thm:transposition}; compare to \cite[Proposition 2.2]{Looper}. In what follows, $v_\p$ denotes the valuation associated to a prime $\p$ of $R$.
\begin{theorem}{\label{thm:dyntransposition}} Let $R$ be a Dedekind domain with quotient field $K$, let 
\[\phi(x)=x^d+Ax^s+B\] 
be a trinomial over $R$ with $d>s\geq1$, and suppose that $\phi$ satisfies the following conditions:  
\begin{enumerate}[topsep=7pt, partopsep=7pt, itemsep=7pt] 
\item[\textup{(1)}] $R$ contains all the critical points of $\phi$, 
\item[\textup{(2)}] $\phi^n$ is irreducible over $K$,    
\item[\textup{(3)}] $ds(d-s)\neq0$ in $R$, 
\item[\textup{(4)}] $\gcd(d,s)=1$.    
\end{enumerate}
For $n\geq2$, let $\alpha$ be a root of $\phi^{n-1}(x)$. If for some critical point $\gamma$ of $\phi$ such that $e(\gamma,\phi)=1$, there exists a prime $\mathfrak{p}_n\subset R$ satisfying:
\begin{enumerate}[topsep=7pt, partopsep=8pt, itemsep=8pt] 
\item[\textup{(a)}] $v_{\mathfrak{p}_n}(\phi^n(\gamma))$ is odd,    
\item[\textup{(b)}] $v_{\mathfrak{p}_n}\big(dA\,\phi^n(\delta)\,\phi^n(0)\big)=0$ for all other critical points $\delta\neq\gamma$ of $\phi$,\vspace{.1cm}    
\end{enumerate}  
then $\Gal_{K(\alpha)}(\phi(x)-\alpha)$, when viewed as a subgroup of $S_d$, contains a transposition. 
\end{theorem} 
\begin{proof} Note first that $N_{K(\alpha)/K}(\phi(\gamma)-\alpha)=\phi^n(\gamma)$, from which it follows that there exists a prime $\mathfrak{q}$ of $S$, the integral closure of $R$ in $K(\alpha)$, such that $v_{\mathfrak{q}}(\phi(\gamma)-\alpha)$ is odd: otherwise, 
\[(\phi(\gamma)-\alpha)S=\mathfrak{q}_1^{s_1}\mathfrak{q}_2^{s_2}\dots \mathfrak{q}_m^{s_m}\] 
for some prime ideals $\mathfrak{q}_i\subset S$ and even exponents $s_i$. In particular, we can apply the ideal norm $N_{S/R}: \mathcal{I}_S\rightarrow\mathcal{I}_R$, a group homomorphism from the ideal group $\mathcal{I}_S$ of $S$ to the ideal group $\mathcal{I}_R$ of $R$, and see that \vspace{.1cm} 
\begin{equation*}
\begin{split} 
\phi^n(\gamma)R&=N_{K(\alpha)/K}(\phi(\gamma)-\alpha)R=N_{S/R}\big((\phi(\gamma)-\alpha)S\big)\\[4pt]
&=N_{S/R}(\mathfrak{q}_1)^{s_1}N_{S/R}(\mathfrak{q}_2)^{s_2}\dots N_{S/R}(\mathfrak{q}_m)^{s_m}\\[4pt]
&=\mathfrak{p}_1^{(s_1f_1)}\mathfrak{p}_2^{(s_2f_2)}\dots \mathfrak{p}_m^{(s_mf_m)}
\end{split} 
\end{equation*}
for some prime ideals $\mathfrak{p}_j$ of $R$ and some integers $f_j$; see \cite[I \S7, Proposition 22]{Lang}. However, this contradicts the fact that there exists $\mathfrak{p}_n\subset R$ such that $v_{\mathfrak{p}_n}(\phi^n(\gamma))$ is odd. Therefore, we can fix a prime $\mathfrak{q}$ of $S$ lying over $\mathfrak{p}_n$ such that $v_{\mathfrak{q}}(\phi(\gamma)-\alpha)$ is odd. 

On the other hand, a similar argument implies that $v_{\mathfrak{q}}(\phi(\delta)-\alpha)=0$ for all critical points $\delta\neq\gamma$, since $N_{K(\alpha)/K}(\phi(\delta)-\alpha)=\phi^n(\delta)$ and $v_{\mathfrak{p}_n}(\phi^n(\delta))=0$ by condition (b). Moreover, \vspace{.1cm}
\begin{equation}{\label{discmonodromy}} 
\Disc(\phi(x)-\alpha)=(-1)^{(d-1)(d-2)/2}d^d\prod_{\delta\in R_\phi}(\alpha-\phi(\delta))^{e(\delta,\phi)}, 
\end{equation}   
where $R_\phi$ is the ramification locus of $\phi$ and $e(\delta,\phi)$ denotes the multiplicity of the critical point $\delta$; see \cite{Hajir}. Hence, (\ref{discmonodromy}) implies that \[v_{\mathfrak{q}}(\phi(\gamma)-\alpha)=v_{\mathfrak{q}}(\Disc(\phi(x)-\alpha))\] 
is odd, since $e(\gamma,\phi)=1$ and $v_{\mathfrak{p}_n}(d)=0$ by assumption (b). 

From here, Capelli's Lemma \cite[Lemma 0.1]{Capelli} implies that $\phi(x)-\alpha$ is irreducible over $K(\alpha)$, since $\phi^n(x)$ is irreducible over $K$. Therefore, if $M_\alpha$ is a splitting field of $\phi(x)-\alpha$ over $K(\alpha)$, then $\mathfrak{q}$ must ramify in $M_\alpha$; here we use that the discriminant of the extension $M_\alpha/K(\alpha)$ and the discriminant of the polynomial $\phi(x)-\alpha$ differ by a square in the coefficient field $K(\alpha)$ (see \cite[III \S3]{Lang}), that a prime of $K(\alpha)$ ramifies in $M_\alpha$ if and only if it divides the discriminant of $M_\alpha/K(\alpha)$ (see \cite[III Corollary 2.12]{Neu}), and that $v_{\mathfrak{q}}(\Disc(\phi(x)-\alpha))$ is odd. Finally, note that $v_{\mathfrak{q}}(A(B-\alpha))=0$, since 
\[N_{K(\alpha)/K}(B-\alpha)=\phi^{n-1}(B)=\phi^n(0)\] 
and $v_{\mathfrak{p}_n}\big(A\,\phi^n(0)\big)=0$ by condition (b). Therefore, conditions (3), (4) and Theorem \ref{thm:transposition} imply that $\Gal\big(M_\alpha/K(\alpha)\big)=\Gal_{K(\alpha)}(\phi(x)-\alpha)$ contains a transposition as claimed.                      
\end{proof}   
In addition to an arithmetic condition on critical orbits that guarantees the existence of transpositions in dynamical Galois groups, we have the following ``maximality criterion"; compare to \cite[Proposition 2.3]{Looper} and \cite[Theorem 3.3]{Jones}. 
\begin{proposition}{\label{prop:maximal}} Suppose that $\phi$, $\gamma$, and $\mathfrak{p}_n$ satisfy the conditions of Theorem \ref{thm:transposition}. Moreover, assume in addition that $\mathfrak{p}_n$ has the following properties: 
\begin{enumerate}[topsep=7pt, partopsep=7pt, itemsep=7pt] 
\item[\textup{(1)}] $\mathfrak{p_n}$ is a primitive prime divisor, i.e. $v_{\mathfrak{p}_n}(\phi^m(\gamma))=0$ for all $1\leq m\leq n-1$. 
\item[\textup{(2)}] If $\delta\neq\gamma$ is any other critical point of $\phi$, then $v_{\mathfrak{p}_n}(\phi^m(\delta))=0$ for all $1\leq m\leq n$. 
\end{enumerate} 
Then $\Gal_{K(\alpha)}(\phi(x)-\alpha)\cong S_d$ implies that $\Gal(K_n/K_{n-1})\cong (S_d)^{d^{n-1}}$. 
\end{proposition}  
\begin{proof} The proof follows that of \cite[Proposition 2.3]{Looper} verbatim; there is nothing in the argument given there that depends upon the characteristic of the ground field.  
\end{proof}  
We are now ready to apply the general tools developed in this section to the Galois groups of iterates of trinomials from the introduction.
\begin{proof}[(Proof of Theorem \ref{thm:eg})] Let $K=k(t)$ and note that $[k(t):k(B)]=\deg(B)\neq0$ is finite. Therefore, it follows from \cite[\S14.4 Proposition 19]{DF} that \vspace{.025cm} 
\[\big[\Aut(T_{p,n}): G_{K,n}(\phi_{p,B})\big]\leq \big[\Aut(T_{p,n}):G_{k(B),n}(\phi_{p,B})\big]\cdot \deg(B)\vspace{.025cm}\]
for all $n\geq1$. Hence, it suffices to prove that $G_{k(B)}(\phi_{p,B})\leq \Aut(T_p)$ is a finite index subgroup, to prove that $G_K(\phi_{p,B})\leq \Aut(T_p)$ is a finite index subgroup. In particular, we may assume that $B=t$ and write $\phi_{p,B}=\phi_p$. Now, let $K_n$ be the splitting field of $\phi_{p}^n$ over $K$ for any $n\geq1$. Since the subgroup of the automorphism group of $T_{p,n}$ (the level $n$ regular $p$-ary rooted tree) that fixes $T_{p,n-1}$ is isomorphic to $(S_p)^{p^{n-1}}$, it suffices to prove that $\Gal(K_n/K_{n-1})\cong (S_p)^{p^{n-1}}$ for all $n$ sufficiently large to prove that $G_K(\phi_{p})\leq \Aut(T_p)$ is a finite index subgroup. To do this, we use Theorem \ref{thm:dyniso}, Theorem \ref{thm:dyntransposition}, and Proposition \ref{prop:maximal}. In particular, we must verify their relevant conditions. 
\noindent First note that $\phi_{p}$ has two distinct critical points: $0$ and 
\[\gamma=-(p-1)\bigg(\frac{-pt^p-pt}{pt^{p-1}+p-1}\bigg).\] 
Here we use our assumption that $p-1$ is non-zero in $K$. Moreover, we see that $e(\phi_{p},\gamma)=1$, i.e. $\gamma$ is a critical point of multiplicity $1$. Most importantly for our purposes, the critical points of $\phi_{p}$ have the following special dynamical property: 
\begin{equation}\label{critical}
\boxed{\phi_{p}^2(0)=\gamma.}
\end{equation}  
Hence, after two applications of $\phi_{p}$, the orbit of $0$ follows that of $\gamma$. On the other hand, one readily computes that $\deg(\phi_{p}^n(0))=p^{n-1}$. Therefore, $0$ and $\gamma$ are both $\phi_p$-wandering points. Likewise, the quantities $\delta_{\phi_{p}}$ and $\epsilon_{\phi_{p}}$ are non-zero since $p(p-1)\in K^*$ by assumption. Therefore, Theorem \ref{thm:dyniso} applied to the pair of points $a=0$ and $b=\gamma$, implies that $\phi_{p}^{n}(\gamma)$ has a primitive prime divisor $\p_n$ appearing to odd valuation for all $n$ sufficiently large (compare to \cite[Theorem 1]{primdiv} for $\ell=2$). Moreover, since we can choose $n$ large enough so that $\p_n$ avoids any finite set of primes (or valuations), we can assume that 
\begin{equation}{\label{valuation}}
v_{\p_n}(\gamma)=0=v_{\p_n}(t);
\end{equation}
see the proof of Theorem \ref{thm:dyniso} or the proofs of \cite[Theorem 1]{primdiv} and \cite[Theorem 1]{Trans}. On the other hand, note that $\phi_{p}$ is Eisenstein at the prime $\p=t$, from which it follows that all iterates of $\phi_{p}$ are irreducible over $K$. Therefore, the hypothesis of Theorem \ref{thm:dyntransposition} for the prime $\p_n\mid \phi_p^n(\gamma)$ are satisfied: $v_{\p_n}(\phi_{p}^n(0))=0$,  since $\p_n$ is primitive for $\phi_{p}^n(\gamma)$ and $\phi_{p}^{n}(0)=\phi_{p}^{n-2}(\gamma)$ by (\ref{critical}). Therefore, $\Gal_{K(\alpha)}(\phi_{p}(x)-\alpha)$ contains a transposition for all roots $\alpha$ of $\phi_{p}^{n-1}(x)$. Hence, $\Gal_{K(\alpha)}(\phi_{d}(x)-\alpha)\cong S_p$ by Proposition \ref{suffconds}. Finally, it follows from Proposition \ref{prop:maximal}, the relationship in (\ref{critical}), the fact that $\p_n$ is primitive, and (\ref{valuation}) that $\Gal(K_n/K_{n-1})\cong (S_p)^{p^{n-1}}$ as claimed. Therefore, $G_K(\phi_{p})\leq \Aut(T_p)$ is a finite index subgroup.                           
\end{proof} 
However, when the ground field $k$ has characteristic zero, we can strengthen our results for $\phi_{p,B}$ and prove surjectivity when $B=t$. 
\begin{proof}[(Proof of Theorem \ref{thm:Odoni})] Let $r=(pt^{p-1}+p-1)^{-1}$ and write $\phi_p^n(\gamma)=a_n/r^{d^n}$ for some $a_n\in \mathbb{Z}[t]$ coprime to $r$ for all $n\geq1$. As in the proof of Theorem \ref{thm:eg}, it suffices to prove that the square-free part of $a_n$ is not completely supported on the primes dividing any of the elements in the set $S_n=\{a_1,a_2,\dots,a_{n-1},t,\gamma\}$ for all $n\geq2$, to prove that $G_K(\phi_{p})= \Aut(T_p)$; this step follows from Theorem \ref{thm:dyntransposition} and Proposition \ref{prop:maximal}. However, it is straightforward to check that the image of $a_n$ in $\mathbb{F}_p[t]$, denoted $\overline{a_n}$, satisfies:
\[-\overline{a_n}=(((t^p+t)^p+t)^p+\dots t)^p+t.\] 
Consequently, $\overline{a_n}$ must be square-free in $\mathbb{F}_p[t]$ (its derivative is a non-zero constant) for all $n\geq1$; see \cite[\S13.5 Proposition 33]{DF}. Hence, $a_n$ is square-free in $\mathbb{Z}[t]$ for all $n\geq1$, since the leading term of $a_n$ is coprime to $p$. Therefore, if $G_K(\phi_{p})\neq\Aut(T_p)$, then $a_n$ itself must be supported on the primes dividing any of the elements of $S_n$ for some $n\ge2$. However, for such $n$, we have the trivial degree bound 
\begin{equation}
\begin{split} 
p^{n+1}=\deg(a_n)&\leq\sum_{m=1}^{n-1}\deg(a_m)+\deg(t)+\deg(\gamma)\\ 
&=p^n+p^{n-1}+\dots+p+1=\frac{p^{n+1}-1}{p-1},
\end{split} 
\end{equation} 
a contradiction. Hence, the arboreal representation of $\phi_p$ is surjective as claimed.            
\end{proof}
\begin{rmk}{\label{rmk:intrinsic}} We note that in contrast to the family of trinomials in \cite[Theorem 1.2]{Looper}, both critical orbits in the family $\phi_{p,B}$ are infinite. Thus, it is unlikely with current techniques that one can establish Theorem \ref{thm:Odoni} via specialization to the number field setting (where effective forms of the Mordell Conjecture are not known). Indeed, if $\phi_{p,t}$ specializes to a polynomial of the form $f_{p,k}(x)=x^p-kx^{p-1}+k$ studied in \cite[Theorem 1.2]{Looper} for some $t\in\mathbb{Q}$, then $t=0=k$ and the arboreal representation of $f_{p,0}(x)=x^p$ is never surjective. Therefore, to prove surjectivity in the family $\phi_{p,t}$ over function fields, one is likely forced to use geometric tools that are intrinsic to this setting, as we have done here.                 
\end{rmk} 
\section{Counting surjective arboreal representations}
Finally, in this section we prove that ``most" (as defined in the introduction) quadratic polynomials over $\mathbb{Z}[t]$ furnish surjective arboreal representations.   
\begin{proof}[(Proof of Theorem \ref{thm:count})] Let $\gamma(t)=a_dt^d+\dots +a_0$ and $c(t)=b_dt^d+\dots +b_0$ for some $a_i,b_j\in\mathbb{Z}$. We think of the coefficients as specializations of the variables $\textbf{a}_i$ and $\textbf{b}_j$ and therefore identify 
\[\mathcal{P}_d\times \mathcal{P}_d=\mathbb{A}^{d+1}\times\mathbb{A}^{d+1}\cong\mathbb{A}^{2d+2}.\] 
In particular, via this identification, we can view the variety 
\[V_n(d):=\big\{(\gamma,c)\in\mathcal{P}_d\times \mathcal{P}_d\,:\, \disc(\phi_{(\gamma,c)}^n(\gamma))=0\big\}\subseteq \mathbb{A}^{2d+2}\]  
as an algebraic subset of $\mathbb{A}^{2d+2}$ for any $n\geq1$; here we use that the discriminant of a polynomial in $t$ is a polynomial in its coefficients, and hence 
\[\disc(\phi_{(\gamma,c)}^n(\gamma))\in\mathbb{Z}[\textbf{a}_{d},\dots \textbf{a}_0,\textbf{b}_0,\dots \textbf{b}_d]\] 
for all $n\geq1$. Moreover, for reasons that will become clear, we define 
\[V_d:=\bigcup_{n=1}^{16}V_n(d)\subset\mathbb{A}^{2d+2}\] for any fixed $d\geq1$. Similarly, define a subset $M_d\subset\mathbb{A}^{2d+2}$ given by  
\[M_d:=\big\{(a_d,\dots,a_0,b_0,\dots b_d)\in\mathbb{A}^{2d+2}\,:\, a_db_d(a_d-b_d)=0\big\}.\\[1pt]\]
We note in Lemma \ref{nonempty} below that the particular element $(t^d,2t^d+t)\in\mathbb{A}^{2d+2}$ is not contained in $M_d\cup V_d$, from which we deduce that
\begin{equation}{\label{openset}}
O_d:=\mathbb{A}^{2d+1}\mysetminus\big(M_d\cup V_d\big)
\end{equation} 
is a non-emty open subset. Therefore, the large sieve inequality \cite[Theorem 3.4.4]{Serre-Galois} implies 
\begin{equation}{\label{eq:count}}
\#\big\{(\gamma,c)\in O_d \,:\,\gamma,c\in\mathcal{P}_d(B)\big\}=(2B+1)^{2d+2}+O(B^{2d+2-1/2}\log B).\vspace{.1cm}
\end{equation}
Here we use that $M_d$ and $V_d$ are ``thin" subsets of $\mathbb{A}^{2d+2}$, since they are both not Zariski dense; see \cite[Definition 3.1.1]{Serre-Galois}. 
\begin{lemma}{\label{nonempty}} Let $d\geq1$, let $\gamma(t)=t^d$, and let $c(t)=2t^d+t$. Then $(\gamma,c)\not\in V_d\cup M_d$. 
\end{lemma} 
\begin{proof}[(Proof of Lemma \ref{nonempty})] It is clear that $(\gamma,c)\not\in M_d$. To show that $(\gamma,c)\not\in V_d$, suppose that some element of the critical orbit is not square-free: $\phi_{(\gamma,c)}^n(\gamma)=f(t)^2\cdot g(t)$ for some $f,g\in\mathbb{Q}[t]$ with $\deg(f)>0$ and some $n\geq1$. By Gauss' lemma, we can assume that $f,g\in\mathbb{Z}[t]$. Moreover, one easily checks that $\phi_{(\gamma,c)}^n(\gamma)$ is monic for all $n\geq2$, and so we may assume that $f$ is monic; when $n=1$, the polynomial $\phi_{(\gamma,c)}^n(\gamma)=2t^d+t$ clearly has non-zero discriminant, and so we need not worry about this case. In particular, the polynomial obtained by reducing the coefficients of $f$ modulo $2$, denoted $\overline{f}\in\mathbb{F}_2[t]$, satisfies $\deg(\overline{f})=\deg(f)>0$. However, the derivative of $\phi_{(\gamma,c)}^n(\gamma)$, when viewed as an element of $\mathbb{F}_2[t]$, is  
\begin{align*}
1\equiv2t^{d-1}+1&\equiv\frac{d}{dt}\Big(\big(\phi_{(\gamma,c)}^{n-1}(\gamma)-t^d\big)^2+2t^d+t\Big)\equiv\frac{d}{dt}\big(\phi_{(\gamma,c)}^{n}(\gamma)\big)\equiv\frac{d}{dt}\big(\,\overline{f}^{\;2}\cdot \overline{g}\big)\\[7pt]
&\equiv\overline{f}^{\;2}\cdot \frac{d}{dt}(\overline{g})+ 2\overline{f}\cdot \frac{d}{dt}(\overline{f})\cdot \overline{g}\equiv\overline{f}^{\;2}\cdot \frac{d}{dt}(\overline{g}). \vspace{.2cm}  
\end{align*}
Hence, $\overline{f}\cdot\overline{f}$ is a unit in $\mathbb{F}_2[t]$, a contradiction of the fact that $\deg(\overline{f})>0$. Therefore, $(\gamma,c)=(t^d,2t^d+t)\not\in V_d$ as claimed.         
\end{proof}
The rest of the proof of Theorem \ref{thm:count} follows from (\ref{eq:count}) and a slightly altered version of our uniformity Theorem in \cite{uniform} applied to the function field $K=\mathbb{Q}(t)$. 
\end{proof}  
\begin{theorem}{\label{thm:translate}} Let $O_d$ be as in (\ref{openset}). If $(\gamma,c)\in O_d$, then $G_K(\phi_{(\gamma,c)})=\Aut(T_2)$.  
\end{theorem}
\begin{proof} 
To establish Theorem \ref{thm:translate}, we translate the proof of \cite[Theorem 1.1]{uniform} part $(1)$ to the slightly weaker assumption that $(\gamma,c)\in O_d$ for some $d\geq1$.
In particular, we first estimate the degree (height) of an element of the critical orbit. Specifically, assume that $(\gamma,c)\in O_d$ and let $\phi=\phi_{(\gamma,c)}$. Then in our notation from the proof of \cite[Theorem 1.1]{uniform}, we have that 
\[h(\phi):=\max\{\deg(\gamma),\deg(c)\}=d=\deg(c-\gamma), \]
since $(\gamma,c)\not\in M_d$. In particular, it is straightforward (\cite[Lemma 1.3]{uniform} part $1$) to see that 
\begin{equation}{\label{deg}} 
\deg(\phi^n(\gamma))=d\cdot 2^{n-1}\;\;\;\text{and}\;\;\;\deg(\phi^n(0))=d\cdot 2^n \\[1pt]
\end{equation} 
for all $n\geq1$. We now use effective height bounds for rational points on curves defined over function fields \cite{Mason,Schmidt} to prove surjectivity. In particular, we prove that $\phi$ defined by $(\gamma,c)\in O_d$ is stable, i.e. all iterates are irreducible.
\begin{lemma} If $(\gamma,c)\in O_d$, then $\phi_{(\gamma,c)}$ is stable over $K=\mathbb{Q}(t)$. 
\end{lemma}
\begin{proof} Suppose that $\phi$ is not stable. Then \cite[Proposition 4.2]{Jones} implies that $\phi^n(\gamma)$ is a square in $\mathbb{Q}(t)$, hence a square in the polynomial ring $\mathbb{Q}[t]$. Note that if $n\leq16$ (in particular $n=1$), then $\phi^n(\gamma)$ is square-free by our assumption that $(\gamma,c)\not\in V_d$. Therefore, if $\phi$ is not stable, there exists $n\geq17$ such that $\phi^n(\gamma)=y_n^2$ for some $y_n\in\mathbb{Q}[t]$. In particular, we obtain an integral point 
\[(X,Y)=\big(\phi^{n-1}(\gamma),\, y_n\cdot(\phi^{n-2}(\gamma)-\gamma)\big)\;\;\;\text{on}\;\; E_\phi: Y^2:=(X-c)\cdot\phi(X),\]
an elliptic curve; $E_\phi$ is non-singular since $c\cdot\phi(c)\neq0$ by (\ref{deg}). Then, as in our proof of stability in \cite[Proposition 1.6]{uniform}, we see that \cite[Theorem 6]{Mason}, \cite[Lemma H]{Schmidt} and (\ref{deg}) together imply that 
\[(2^{n-2}) d\leq 110d+4.\]
Hence, $n\leq8$ (independent of $d$), a contradiction. Therefore, $\phi$ must be stable.         
\end{proof} 
Now that we know that $\phi$ is stable, it suffices to show that $\phi^n(\gamma)$ has primitive prime divisors of odd valuation for all $n\geq2$, to prove Theorem \ref{thm:translate}; see \cite[Theorem 3.3]{Jones}. In particular, if $\phi^n(\gamma)$ does not have such primitive divisors, then there exist $d_n, y_n\in\mathbb{Q}[t]$ such that  
\begin{equation}{\label{decomp:sqfree}}
\phi^n(\gamma)=d_n\cdot y_n^2\;\;\; \text{and}\;\;\; d_n=\prod p_j^{e_j},    
\end{equation} 
where the $p_j\in\mathbb{Q}[t]$ are irreducible polynomial dividing some $\phi^{t_j}(\gamma)$ with $1\leq t_j\leq n-1$, and the $e_j$ are exponents satisfying $e_j\in\{0,1\}$. Concretely, we can decompose $\phi^n(\gamma)$ into a square and square-free part, and if $\phi^n(\gamma)$ does not have primitive prime divisors with odd valuation, then the square-free part of $\phi^n(\gamma)$ is comprised of primes dividing lower order iterates. 

On the other hand, since each $p_j$ divides both $\phi^n(\gamma)$ and $\phi^{t_j}(\gamma)$, it must divide $\phi^{n-t_j}(0)$ also. Therefore, for each $j$ we can choose $s_j:=\min\{t_j,n-t_j\}$ so that $d_n=\prod p_j^{e_j}$ with 
\begin{equation}{\label{refinement}}
p_j\big\vert\phi^{s_j}(\gamma)\;\;\text{or}\;\;p_j\big\vert\phi^{s_j}(0),\;\;\;\;\;\;\text{for some}\;\;1\leq s_j\leq \Big\lfloor\frac{n}{2}\Big\rfloor.
\end{equation}
Therefore, we deduce from (\ref{refinement}) that 
\begin{equation}{\label{estimate1}}
\deg(d_n)\leq\bigg(\sum_{j=1}^{\lfloor\frac{n}{2}\rfloor} \deg(\phi^{j}(\gamma))+ \sum_{i=1}^{\lfloor\frac{n}{2}\rfloor} \deg(\phi^{i}(0))\bigg)=d\bigg(\sum_{j=1}^{\lfloor\frac{n}{2}\rfloor}2^{j-1} + \sum_{i=1}^{\lfloor\frac{n}{2}\rfloor}2^i\bigg).\\[5pt]
\end{equation}
On the other hand the decomposition on (\ref{decomp:sqfree}) determines an integral point 
\begin{equation}{\label{twistedcurve}}
(X,Y)=\big(\phi^{n-1}(\gamma),\, y_n\cdot(\phi^{n-2}(\gamma)-\gamma)\big)\;\;\;\text{on}\;\; E_\phi^{(d_n)}: Y^2:=d_n(X-c)\cdot\phi(X),
\end{equation}  
a quadratic twist of the elliptic curve $E_\phi$. In particular, \cite[Theorem 6]{Mason}, (\ref{deg}), (\ref{estimate1}) and (\ref{twistedcurve}) together imply that 
\[\frac{2^{n-2}}{2^{\lfloor\frac{n}{2}\rfloor+1}+2^{\lfloor\frac{n}{2}\rfloor}+1}\leq36;\]
this estimate follows the argument given in \cite[Theorem 1]{uniform} part (1) verbatim. Hence, $n\leq16$, independent of $d$.  

However, $\phi^n(\gamma)$ is square-free for all $n\leq16$, since $(\gamma,c)\not\in V_d$ by assumption. In particular, $y_n\in\mathbb{Q}$, i.e. the square-part in the decomposition in (\ref{decomp:sqfree}) must be trivial. Hence, if $\phi^n(\gamma)$ does not have primitive prime divisors with odd valuation for some $n\geq2$, then (\ref{deg}) and (\ref{decomp:sqfree}) imply that
\[(2^{n-1})d=\deg(\phi^n(\gamma))=\deg(d_n)\leq (1+2+\dots +2^{n-2})d=(2^{n-1}-1)d,\]
since $d_n$ is square-free and supported on primes dividing $\phi^m(\gamma)$ for $1\leq m\leq n-1$. However, this inequality forces $d=0$, a contradiction.              
\end{proof} 

\end{document}